\documentclass[graybox]{svmult}
\usepackage{amsmath,amssymb,amsthm}
\usepackage{mathtools}
\usepackage{csquotes}
\usepackage{mathtools}
\usepackage{mathrsfs}
\usepackage[colorlinks=true,linkcolor=blue,urlcolor=blue,citecolor=blue]{hyperref}
\usepackage{pgfplots}
\pgfplotsset{compat=1.18}
\usepackage{caption}
\theoremstyle{definition}
\theoremstyle{remark}

\title*{The P\'olya-Szeg\"{o} principle in the fractional setting: a glimpse on nonlocal functional inequalities}
\titlerunning{The fractional P\'olya-Szeg\"o Principle}
\author{Alessandro Carbotti\orcidID{0000-0001-9394-8172}}
\institute{Alessandro Carbotti \at Dipartimento di Matematica e Fisica ``Ennio De Giorgi'', Università del Salento, Via per Arnesano, 73100, Lecce, Italy \email{alessandro.carbotti@unisalento.it}}

\begin{document}
	
	\maketitle
	
\abstract*{In this chapter we present the fractional Pólya–Szegő principle and its main consequences in
	the study of nonlocal functional inequalities. In particular, we show how
	symmetrization methods work also in the fractional setting and yield sharp results
	such as isoperimetric type inequalities. Further developments including stability
	issues and generalizations in the anisotropic and the Gaussian setting are also discussed.}
	
\abstract{In this chapter we present the fractional Pólya–Szegő principle and its main consequences in
	the study of nonlocal functional inequalities. In particular, we show how
	symmetrization methods work also in the fractional setting and yield sharp results
	such as isoperimetric type inequalities. Further developments including stability
	issues and generalizations in the anisotropic and the Gaussian setting are also discussed.}

	\section{Introduction}
	
	Fractional Sobolev spaces and nonlocal operators have gained significant attention in the last decades. One reason is their natural appearance in various fields, including probability theory, mathematical physics, harmonic analysis, and geometric measure theory. These nonlocal models provide more accurate descriptions of phenomena characterized by long-range interactions or anomalous diffusion, such as in Lévy processes, materials with memory, image processing, and finance, see e.g. \cite{BucVal16, CDV19, Garofalo1, GilOsh}.
	
	Functional inequalities play a pivotal role in understanding the qualitative and quantitative behaviour of solutions to partial differential equations. They provide essential tools for establishing existence, regularity, stability, and concentration properties of solutions, as well as driving the formulation of variational problems and numerical schemes.
	
	In the local framework functional inequalities have deep geometric content and powerful analytic consequences. Moving to the nonlocal setting many of the aforementioned features are inherited, though additional issues arise owing to long range interactions. In particular, the lack of locality affects the behaviour of nonlocal functionals under rearrangement.
	
	This chapter focuses on functional inequalities involving geometric quantities enjoying a variational characterization and derived using the \emph{P\'olya-Szeg\"{o} principle} based on the concept of symmetric decreasing rearrangement. 
	
	In the book \cite{PolyaSzego} by G. P\'olya and G. Szeg\"o, the authors prove that for every $1\le p<\infty$ and $u\in W^{1,p}(\mathbb{R}^N,[0,\infty))$ the symmetric decreasing rearrangement $u^*$ (see the following Definition \ref{def:symmdecrearfun}) also belongs to $W^{1,p}(\mathbb{R}^N,[0,\infty))$ and 
	\begin{equation}
		\label{eq:classicalps}
		\int_{\mathbb{R}^N}|\nabla u^*|^pdx\le\int_{\mathbb{R}^N}|\nabla u|^pdx.
	\end{equation}
	Rearrangement techniques allow one to reduce more involved variational problems to radial or one-dimensional ones, and this reduction may simplify the identification of extremal profiles and optimal constants in functional inequalities. 
	
The rest of the chapter is organized as follows: In Section \ref{sec:notprel} we introduce some notations and preliminary definitions. In Section \ref{sec:fracps} we state and prove in \ref{th:frpsi} the P\'olya-Szeg\"o principle. In Section \ref{sec:applications} we recall some of the most famous applications of the P\'olya-Szeg\"o inequality, namely
\begin{enumerate}
	\item The fractional Sobolev inequality in Subsection \ref{sec:fracsoboine} \\
	\item The fractional isoperimetric inequality in Subsection \ref{sec:isopine} \\
	\item The Faber-Krahn inequality for the fractional $p$-laplacian in Subsection \ref{sec:fracfk} \\
	\item A Talenti-type comparison for solutions to fractional equations in Subsection \ref{sec:talenti}.
	\end{enumerate}
In Section \ref{sec:pspviaext} we sketch an alternative proof of Theorem \ref{th:frpsi} in the case $p=2$ based on the Extension Techniques of \cite{CafSil} and partial symmetrization. In Section \ref{sec:stability} the problem of stability of the abovementioned inequalities is treated. In Section \ref{sec:anisotrop} some generalizations involving possibly anisotropic nonlocal functionals are taken into account. The chapter concludes with some final comments and open problems set in Section \ref{sec:openprob}.

	\section{Notations and Preliminaries}
	\label{sec:notprel}
	Let $s \in (0,1)$, $N\in\mathbb{N}$ and $1\le p<\infty$. The fractional Sobolev space $W^{s,p}(\mathbb{R}^N)$ is defined as the space of functions $u \in L^p(\mathbb{R}^N)$ such that the following quantity
	\begin{equation}
		\label{eq:fracseminorm}
		[u]_{W^{s,p}(\mathbb{R}^N)}:=\left(\int_{\mathbb{R}^N}\int_{\mathbb{R}^N} \frac{|u(x)-u(y)|^p}{|x-y|^{N+sp}}\,dx\,dy\right)^{1/p}
	\end{equation}
	is finite. We refer to $[\cdot]_{W^{s,p}(\mathbb{R}^N)}$ as \emph{Gagliardo-Slobodeckij seminorm}.
	
	It is easy to check that the space $W^{s,p}(\mathbb{R}^N)$ endowed with the norm
	
	\begin{equation}
		\label{eq:fracsobnorm}
		\left\|u\right\|_{W^{s,p}(\mathbb{R}^N)}=\left(\left\|u\right\|^p_{L^p(\mathbb{R}^N)}+[u]^p_{W^{s,p}(\mathbb{R}^N)}\right)^{1/p}
	\end{equation}
	is a Banach space. In particular, when $p=2$, $W^{s,2}(\mathbb{R}^N)$ with the norm \eqref{eq:fracsobnorm} is a Hilbert space which we will indicate in the sequel as $H^s(\mathbb{R}^N)$. It is worth mentioning that the Gagliardo seminorms rescaled respectively by a factor $s$ when $s\to 0^+$ or $1-s$ when $s\to 1^-$, approach in the limit the $L^p(\mathbb{R}^N)$ norm or the $W^{1,p}(\mathbb{R}^N)$ seminorm, both in the pointwise and in the $\Gamma$-convergence sense see \cite{BBM01, Davila02, MazSha, Ponce04}. We refer to \cite{DNPV12} for further details and properties.
	
	\subsection{The fractional laplacian}
	
	The prototype of nonlocal fractional operator is the fractional laplacian $(-\Delta)^s$ which is defined as the Fourier multiplier with symbol $|\cdot|^{2s}$, namely
	\begin{equation}
	\label{eq:fraclapfourier}
	(-\Delta)^su(x)=\mathcal{F}^{-1}\left(|\cdot|^{2s}\mathcal{F}(u)\right)(x),
	\end{equation}
	where $\mathcal{F}$ denotes the Fourier transform on $L^1(\mathbb{R}^N)\cap L^2(\mathbb{R}^N)$
	$$
	\mathcal{F}(u)(\xi)=\int_{\mathbb{R}^N}u(x)e^{-ix\cdot\xi}dx.
	$$
	The fractional laplacian enjoys many equivalent definitions on the whole of $\mathbb{R}^N$. We give here three among others which will be useful in the sequel
	\begin{enumerate}
		\item{\textbf{Fractional laplacian as integro-differential operator}}
		
		For $u\in H^s(\mathbb{R}^N)$, the fractional laplacian $(-\Delta)^s u(x)$ is given by
		\begin{equation}
			\label{eq:fraclapsingint}
			(-\Delta)^s u(x) = C_{N,s}\, \text{P.V.} \int_{\mathbb{R}^N} \frac{u(x)-u(y)}{|x-y|^{N+2s}}\,dy,
		\end{equation}
		where P.V. stands for the Cauchy principal value. Moreover, if $u\in H^s(\mathbb{R}^N)\cap C^2(B_R)$ for some $R>0$, the principal value may be dropped and the integral has to be intended in the Lebesgue sense.
				
		\begin{remark}
			The assumption $u\in H^s(\mathbb{R}^N)\cap C^2(B_R)$ can be relaxed by requiring that $u\in C^{0,2s+\varepsilon}(B_R)$ if $0<s<\frac12$ or $u\in C^{1,2s-1+\varepsilon}(B_R)$ for some $\varepsilon>0$ if $\frac12\le s<1$ and that $\frac{u}{1+|\cdot|^{N+2s}}\in L^1(\mathbb{R}^N)$ so that $(-\Delta)^s u(x)$ is still pointwise well-defined almost everywhere.
		\end{remark}
		
		\item{\textbf{Fractional laplacian via heat flow}}
		
		Given a sufficiently smooth function $u$ with a certain growth at infinity, the \emph{heat semigroup} $e^{t\Delta}$ applied to $u$ is given by $U(t,x)=e^{t\Delta}u:=G_t\ast u$ where, for every $t>0$, $G_t$ denotes the \emph{heat kernel}		
		\begin{equation}
			\label{eq:heatkernel}
			G_t(\cdot) := (4\pi t)^{-N/2} e^{-\frac{|\cdot|^2}{4t}}.
		\end{equation}
		It can be proved that $U$ solves the Cauchy problem
		$$
		\begin{cases*}
			U_t=\Delta U\quad\text{in}\quad(0,\infty)\times\mathbb{R}^N \\
			U(0,x)=u(x)\quad\text{in}\quad\mathbb{R}^N.
		\end{cases*}
		$$
		
		Since for every $\lambda>0$ and $s\in(0,1)$ the following identity holds true
		
		$$
			\lambda^s=\frac{1}{\Gamma(-s)}\int_0^\infty\frac{e^{-t\lambda}-1}{t^{s+1}}dt,
		$$
		where $\Gamma$ denotes the Euler's Gamma function, by means of functional calculus one can replace the positive number $\lambda$ with the positive operator $-\Delta$ and consider the fractional powers $(-\Delta)^s$ through the \emph{Bochner Subordination formula}
		\begin{equation}
			\label{eq:Bochner}
			(-\Delta)^su=\frac{1}{\Gamma(-s)}\int_0^\infty\frac{e^{t\Delta}u-u}{t^{s+1}}dt.
		\end{equation}
		We refer e.g. to the book \cite{MarSan} for further details on the definition of fractional powers of operators via functional calculus.
		
		\item{\textbf{Fractional laplacian via ``Harmonic'' extension}}
		
		The fractional laplacian enjoys the following characterization in terms of the Dirichlet-to-Neumann map associated to a divergence form local operator defined on the upper half-space $\mathbb{R}^{N+1}_+$ with degenerate or singular coefficients. Namely
		
		\begin{theorem}(\cite[Section 3]{CafSil})
			\label{th:cafsil}
			Let $u:\mathbb{R}^N\rightarrow\mathbb{R}$ and $U:\mathbb{R}^{N+1}_+ \rightarrow\mathbb{R}$ be the unique weak solution of the problem
			\begin{equation}
				\label{eq:divformeq}
				\begin{cases}
					\text{div}(y^{1-2s}\nabla U)=0\quad\text{in}\quad \mathbb{R}^{N+1}_+ \\
					U(x,0)=u(x)\quad\text{in}\quad\mathbb{R}^N.
				\end{cases}
			\end{equation}
			Then
			$$
			(-\Delta)^su(x)=-c_s\lim_{y\to 0^+}y^{1-2s}\partial_yU(x,y),
			$$
			and
			\begin{equation}
				\label{eq:extseminorm}
				[u]^2_{H^s(\mathbb{R}^N)}=c_s\inf_{\substack{U\in H^1_{\rm loc}(\mathbb{R}^{N+1}_+) \\ U\left(\cdot,0\right)=u}}\int_0^\infty y^{1-2s}dy\int_{\mathbb{R}^N}\left(|\nabla_xU|^2+|\partial_yU|^2\right)dx
			\end{equation}
		\end{theorem}
		\begin{remark}
			In the case $s=\frac12$, Problem \eqref{eq:divformeq} is the harmonic extension on the upper half-space, and the ``half''-laplacian $(-\Delta)^{1/2}$ is the inward normal derivative of the extension $U$ at $\{y=0\}$.
		\end{remark}
		
				\begin{remark}
			The presence of the constant $C_{N,s}:=\frac{4^s\Gamma\left(\frac N2+s\right)}{\pi^{N/2}\Gamma(1-s)}$ in \eqref{eq:fraclapsingint} makes consistent the equivalence among the several definition of $(-\Delta)^s$. See also \cite[Section 3]{BucVal16}.
			\end{remark}
			
		\begin{remark}
			It is worth noticing that, for $1<p<\infty$, one can also consider a nonlinear counterpart of $(-\Delta)^s$ given by the \emph{fractional $p$-laplacian}
			$$
			(-\Delta)^s_pu(x):=\text{P.V.}\int_{\mathbb{R}^N}\frac{|u(x)-u(y)|^{p-2}(u(x)-u(y))}{|x-y|^{N+sp}}dy.
			$$
			We refer to \cite{DGV} for analogous characterization of $(-\Delta)^s_p$ as in the rest of this Subsection.
		\end{remark}
	\end{enumerate}
	
	\begin{definition}[Schwarz symmetrization]
			We define the symmetric decreasing rearrangement (or \emph{Schwarz symmetral}) $\Omega^*$ of an open and bounded set $\Omega$ as the ball centered at the origin with the same measure of $\Omega$, namely 
	$$
	\Omega^*=\{x\in\mathbb{R}^N:\,\;\omega_N|x|^N<|\Omega|\},
	$$
	where $|\cdot|$ stands as usual for the $N$-dimensional Lebesgue measure of a measurable set and $\omega_N$ denotes the volume of the unit ball in $\mathbb{R}^N$.
	\end{definition}
	
		\begin{definition}[Symmetric decreasing rearrangement of a function]
		\label{def:symmdecrearfun}
		Let $u: \mathbb{R}^N \to [0,\infty)$ be a measurable function. We say that $u$ \emph{vanishes at infinity} if
		$$
		|\{u(x)>t\}|<\infty\quad\text{for every}\quad t>0.
		$$
		The \emph{symmetric decreasing rearrangement} $u^*:\mathbb{R}^N\rightarrow[0,\infty)$ of a function $u$ vanishing at infinity
		is defined through the \emph{layer-cake} representation formula as
		$$
		u^*(x)=\int_0^\infty\chi_{\{f(x)>t\}^*}dt.
		$$
		The \emph{distribution function} of $u$ is defined by
		\[
		\mu_u(t) := |\{ x \in \mathbb{R}^N : u(x) > t \}|, \quad t > 0,
		\]
		and $u^*$ is the unique radially symmetric decreasing function which is \emph{equimeasurable} with respect $u$, i.e.
		\[
		\mu_{u^*}(t) = \mu_u(t), \quad \text{for all } t > 0.
		\]
		More explicitly, there exists a function $u_0: [0, \infty) \to [0, \infty)$, decreasing and right-continuous, such that $u^*(x)=u_0(\omega_N|x|^N)$ and, for every $r>0$
		\[
		u_0(r) = \sup \left\{ t \ge 0 : \mu_u(t)>r \right\}.
		\]
\end{definition}

\begin{remark}
		Notice that $\Omega^*$ is the unique set such that $\chi_{\Omega^*}=(\chi_\Omega)^*$.
\end{remark}

\begin{figure}[h]
	\centering
	\includegraphics[width=1.55\textwidth]{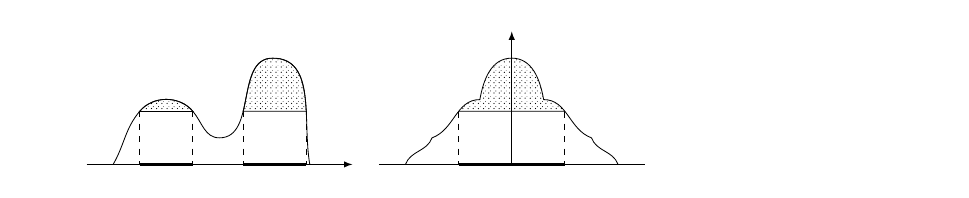}
		\caption[Schwarz rearrangement of a 1D function.]{Schwarz rearrangement of a 1D function.\footnotemark}
	\label{fig:FigureBC_Carbotti}
\end{figure}
\footnotetext{Image kindly provided by Andrea Gentile.}

Two remarkable results involving symmetric decreasing rearrangement are the following
\begin{lemma}[Rearrangement preserves $L^p$ norms]
	\label{lem:layer}
	Let $1\le p\le\infty$. For every non-negative $u\in L^p(\mathbb{R}^N)$ it holds that $u^*\in L^p(\mathbb{R}^N)$ with
	$$
	\left\|u^*\right\|_{L^p(\mathbb{R}^N)}=\left\|u\right\|_{L^p(\mathbb{R}^N)}.
	$$
	\end{lemma}
	\begin{lemma}[Riesz Rearrangement inequality]
		\label{lem:rri}
		Let $f,g,h:\mathbb{R}^N\rightarrow[0,\infty)$ measurable functions. Then
		$$
		\int_{\mathbb{R}^N}f(x)g(x-y)h(y)dxdy\le\int_{\mathbb{R}^N}f^*(x)g^*(x-y)h^*(y)dxdy.
		$$
	\end{lemma}
	
	The Schwarz symmetrization is not the only rearrangement that allows for the P\'olya-Szeg\"o inequality \eqref{eq:classicalps}. We introduce the \emph{Steiner symmetrization} of a set
	
		\begin{definition}[Steiner symmetrization]
		Let $\Omega\subset\mathbb{R}^N$ be a measurable set and let $e \in \mathbb{S}^{N-1}$ be a fixed direction. 
		For each $x \in e^\perp$, consider the one-dimensional section
		\[
		\Omega_x := \{ t \in \mathbb{R} : x + te \in \Omega \}.
		\]
		The \emph{Steiner symmetral} of $\Omega$ with respect to the direction $e$ is the set
		\[
		\Omega^\sharp := \{ x + te : x \in e^\perp, \ |t| < \tfrac{1}{2} \, \mathcal{L}^1\left(\Omega_x\right) \},
		\]
		where $\mathcal{L}^1\left(\Omega_x\right)$ denotes the one-dimensional Lebesgue measure of $\Omega_x$. 
		
		In other words, each section $\Omega_x$ along the line parallel to $e$ is replaced by the centered interval of the same length, symmetric with respect to the hyperplane $e^\perp$.
	\end{definition}

	\begin{definition}[Steiner symmetrization of a function]
		\label{def:steiner}
		Let $f : \mathbb{R}^N \to [0,\infty)$ be a measurable function and let $e \in \mathbb{S}^{N-1}$ be a fixed direction. 
		For each $x \in e^\perp$, consider the one-dimensional function
		\[
		f_x(t) := f(x + te), \qquad t \in \mathbb{R}.
		\]
		We denote by $f_x^\sharp$ the one-dimensional symmetric decreasing rearrangement of $f_x$, i.e., the unique nonnegative, even, and nonincreasing function such that
		\[
		\mathcal{L}^1\left(\{ t \in \mathbb{R} : f_x(t) > \lambda \}\right) = \mathcal{L}^1\left(\{ t \in \mathbb{R} : f_x^\sharp(t) > \lambda \}\right) \quad \text{for all } \lambda > 0.
		\]
		The \emph{Steiner symmetrization} of $f$ with respect to the direction $e$ is then defined by
		\[
		f^\sharp(x + te) := f_x^\sharp(t), \qquad x \in e^\perp, \ t \in \mathbb{R}.
		\]
		
		In other words, along each line parallel to $e$, the function is replaced by its one-dimensional symmetric decreasing rearrangement centered with respect to the hyperplane $e^\perp$.
	\end{definition}
	
		\begin{remark}
		
		We highlight some analogies and differences between Steiner and Schwarz symmetrizations. Both the first and the second one aim to replace a given set or function with a more symmetric one, preserving the volume measure or the $L^p$-norm and proving the monotonicity of Dirichlet energies such as in \eqref{eq:psine} or for convolution-type integrals such as in Lemma \ref{lem:rri}. In the following we summarize the main differences

		\begin{center}
			\begin{tabular}{|p{5.0cm}|p{6.0cm}|p{4.5cm}|}
				\hline
				\textbf{} & \textbf{Steiner Symmetrization} & \textbf{Schwarz Symmetrization} \\
%				\hline
%				Type of symmetry & Reflects symmetry with respect to a fixed hyperplane. & Produces full radial symmetry about a point, usually the origin. \\
				\hline
				Construction for sets & For each line parallel to a fixed direction, replace the 1D section of $E$ by an interval centered on the hyperplane with the same measure. & Replace each section of $E$ by a ball centered at the origin having the same measure. \\
				\hline
				Construction for functions & For each line parallel to a fixed direction, replace the 1D level sets of $u$ by an interval centered on the hyperplane with the same measure. & Replace each level set of $u$ by a ball centered at the origin having the same measure. \\
				\hline
				Symmetry of rearranged functions & 1D even symmetry with respect to a chosen direction. & Radially symmetric. \\
				\hline
				Symmetry of rearranged sets & Axial symmetry with respect to a hyperplane. & Radial symmetry. \\
				\hline
				Iterative use & Repeated Steiner symmetrizations in many directions approximate Schwarz symmetrization. & Fully symmetric after one step. \\
				\hline
				Effect on convexity & Preserves convexity: if $E$ is convex, then its Steiner symmetral is convex. & Does not necessarily preserve convexity (except for balls). \\
				\hline
				Typical applications & Useful in anisotropic problems, or when convexity must be preserved. & Comparison between differential and variational problems with the radial symmetric counterpart. \\
				\hline
%				Notation & $E^{*_{e}}$ or $u^{*_{e}}$, indicating the symmetrization along direction $e$. & $E^*$ or $u^*$, the radially symmetric decreasing rearrangement. \\
%				\hline
			\end{tabular}
		\end{center}
	\end{remark}
	
	\section{The P\'olya-Szeg\"o principle for Gagliardo fractional seminorms}
	\label{sec:fracps}
We are now in the position to state the main result of this chapter
	\begin{theorem}[Fractional P\'{o}lya--Szeg\"o Inequality]
		\label{th:frpsi}
		Let $1\le p<\infty$ and $u \in W^{s,p}(\mathbb{R}^N)$. Then $u^* \in W^{s,p}(\mathbb{R}^N)$ and
		\begin{equation}
			\label{eq:psine}
			[u^*]_{W^{s,p}(\mathbb{R}^N)}\le[u]_{W^{s,p}(\mathbb{R}^N)}.
		\end{equation}
		If $p=1$ then equality holds if and only if $u$ does not change sign and the level set $\{u>\tau\}$ is a ball for almost every $\tau>0$. If $p>1$ equality holds if and only if $u$ is the translate of a symmetric decreasing function.
	\end{theorem}
	\begin{proof}
		 Since $||u(x)|-|u(y)||\le |u(x)-u(y)|$ and equality for every $x,y\in\mathbb{R}^N$ holds if and only if $u\ge 0$ we reduce ourselves to prove the claim for non-negative functions.
		 Moreover, we assume for the moment that $u$ is also bounded. Let $G_t(\cdot)$ be the heat kernel \eqref{eq:heatkernel}.
		The fractional seminorm \eqref{eq:fracseminorm} can be represented (up to a multiplicative constant) as:
		\begin{equation}
			\label{eq:heatrepr}
			\begin{split}
			[u]_{W^{s,p}(\mathbb{R}^N)}^p &= C_{N,s,p} \int_0^\infty t^{-1 - \frac{sp}{2}} \int_{\mathbb{R}^N} \int_{\mathbb{R}^N} |u(x) - u(y)|^p G_t(x - y)\, dxdy\, dt \\
			&=C_{N,s,p} \int_0^\infty I_t[u]t^{ \frac{N+sp}{2}-1}dt,
			\end{split}
		\end{equation}
		where, for every $t>0$ we have set
		$$
		I_t[u]:=\int_{\mathbb{R}^N}\int_{\mathbb{R}^N}|u(x)-u(y)|^pe^{-t|x-y|^2}dxdy.
		$$
		
		Equality \eqref{eq:heatrepr} is obtained via Bochner subordination formula \eqref{eq:Bochner} since
		$$
		\frac{C_{N,s,p}}{|x-y|^{N+sp}}=\int_0^\infty G_t(x-y)t^{-1-\frac{sp}{2}}dt.
		$$
		
		We now prove that, for every $t>0$ rearrangement decreases the functional $I_t[u]$. We write $I_t[u]=I^+_t[u]+I^-_t[u]$, where
		
		$$
			I^{\pm}_t[u]:=C_{N,s,p}\int_{\mathbb{R}^N}\int_{\mathbb{R}^N}(u(x)-u(y))_{\pm}^{p}e^{-t|x-y|^2}dxdy,
		$$
		where, as usual
		$$
		(r)_+:=
		\begin{cases*}
			r\quad\text{if}\quad r\ge 0 \\
			0\quad\text{if}\quad r<0
		\end{cases*}
		$$
		and $(r)_-=(-r)_+$.
		We firstly argue on $I^+_t[u]$. Since the function $\mathbb{R}\ni z\mapsto J(z):=|z|^p$ is convex, so does $J_+$ and it has a non-negative and non-decreasing right derivative $J'_+$, so that we can write
		$$
		(u(x)-u(y))_+^p=\int_0^\infty J_+'(u(x)-\tau)\chi_{\{u\le \tau\}}(y)d\tau,
		$$
		that yields
		\begin{equation}
			\label{eq:Ipiu}
			I^+_t[u]=C_{N,s,p}\int_0^\infty\left(\int_{\mathbb{R}^N}\int_{\mathbb{R}^N}J'_+(u(x)-\tau)e^{-t|x-y|^2}\chi_{\{u\le\tau\}}(y)dxdy\right)d\tau.
		\end{equation}
		Since $u$ is bounded and $|\{u>\tau\}|<\infty$ for every $\tau>0$ we have that
		
		$$
		\int_{\mathbb{R}^N}J'_+(u(x)-\tau)dx<\infty,
		$$
		Therefore, denoting by $e^+_{t,\tau}[u]$ the inner integral in \eqref{eq:Ipiu} and using that $\chi_{\{u\le\tau\}}=1-\chi_{\{u>\tau\}}$, we write
		
		\begin{equation}
			\label{eq:simme}
			\begin{split}
			e^+_{t,\tau}[u]&=\left\|e^{-t|\cdot|^2}\right\|_{L^1(\mathbb{R}^N)}\int_{\mathbb{R}^N}J'_+(u(x)-\tau)dx \\
			&-\int_{\mathbb{R}^N}\int_{\mathbb{R}^N}J'_+(u(x)-\tau)e^{-t|x-y|^2}\chi_{\{u>\tau\}}(y)dxdy.
			\end{split}
		\end{equation}
		We notice that the rearrangement leaves invariant the first term in the right-hand side of \eqref{eq:simme} since $J'_+\left((u(x)-\tau)\right)^*=J'_+(u^*(x)-\tau)$, while the second term does not decrease thanks to Lemma \ref{lem:rri}. This proves that $e^+_{t,\tau}[u]\ge e^+_{t,\tau}[u^*]$ for almost every $\tau >0$ and for every $t>0$. Integrating in $\tau\in(0,\infty)$ 
		we conclude that
		$$
		I^+_t[u]\ge I^+_t[u^*]
		$$
		for every $t>0$. Arguing in the same way for $I^-_t[u]$ exchanging the roles of $x$ and $y$ in the definition of $e^+_{t,\tau}[u]$ and replacing $(\cdot)_+$ with $(\cdot)_-$, we conclude that
		
		\begin{equation}
			\label{eq:rearrangementIt}
			I_t[u]\ge I_t[u^*]
		\end{equation}
		for every $t>0$.
		
		Multiplying by $t^{\frac{N+sp}{2}-1}$, integrating in $t \in (0, \infty)$ and raising to the power $1/p$ both sides of \eqref{eq:rearrangementIt} we obtain:
		\[
		[u^*]_{W^{s,p}(\mathbb{R}^N)} \leq [u]_{W^{s,p}(\mathbb{R}^N)},
		\]
		which concludes the first part of the proof.
		
		Now, we characterize the case of equality in \eqref{eq:psine}.
		
		Assume that $I^\pm_t[u^*]=I^\pm_t[u]$ for some $u\in W^{s,p}(\mathbb{R}^N)\cap L^\infty(\mathbb{R}^N)$ and every $t>0$. By \eqref{eq:simme} we have that $e^\pm_{t,\tau}[u^*]=e^\pm_{t,\tau}[u]$ for every $t>0$ and almost every $\tau>0$. By Lieb's strict rearrangement inequality \cite[Lemma 3]{lieb}, we have that for a.e. $\tau>0$ there exists $a_\tau\in\mathbb{R}^N$ such that $\chi_{\{u<\tau\}}(x)=\chi_{\{u^*<\tau\}}(x-a_\tau)$ and
		$$
		J'_{\pm}(u(x)-\tau)=J'_{\pm}(u^*(x-a_\tau)-\tau).
		$$
		Now, if $p=1$ then $J_+(t)=t_+$ for every $t\in\mathbb{R}$ and this means that the superlevel sets $\{u>\tau\}$ are balls for almost every $\tau>0$. Otherwise if $p>1$ then $J_+$ is strictly convex in $(0,\infty)$ and so $J'_+$ is strictly increasing in $(0,\infty)$ which yields
		\begin{equation}
			\label{eq:indiptau}
			u(x)=u^*(x-a_\tau)
		\end{equation}
		for almost every $x\in\mathbb{R}^N$, $\tau>0$.
		Since the left-hand side in \eqref{eq:indiptau} does not depend on $\tau$ we have that $u$ is a translate of a symmetric decreasing function, and this concludes the proof for $u\in L^\infty(\mathbb{R}^N)$.
		
		Now we remove the assumption that $u$ is bounded, that is, we
		claim that \eqref{eq:rearrangementIt} holds for any non–negative $u$ with $I_t[u]$ and $|\{u>\tau\}|$ finite for almost every
		$\tau>0$ and every $t>0$. To see this, replace $u$ by $u_{M}=\min\{u,M\}$ and note that $(u_{M})^{*}=(u^{*})_{M}=:u^{*}_{M}$
		and $I_t[u_{M}]\le I_t[u]$. By monotone convergence the claim follows easily from the inequality
		\begin{equation}
			\label{eq:removebound}
			I_t[u_{M}]\ge I_t[u^{*}_{M}].
		\end{equation}

		Finally, we characterize the cases of equality for general $u$.
		Assume that $I_t^{+}[u]=I_t^{+}[u^{*}]$, for any $M>0$ we decompose
		\[
		u=u_{M}+v_{M},\qquad u_{M}:=\min\{u,M\},
		\]
		and find
		\begin{equation}
			\label{eq:A.7}
			I_t^{+}[u]=I_t^{+}[u_{M}]+I_t^{+}[v_{M}]
			+\iint_{\mathbb{R}^{N}\times\mathbb{R}^{N}}
			F_{M}\big(v_{M}(x),u_{M}(y)\big)e^{-t|x-y|^2}\,dx\,dy
		\end{equation}
		with
		\[
		F_{M}(v,u):=J_{+}(v+M-u)-J_{+}(v)-J_{+}(M-u).
		\]
		Note that since $J_{+}$ is convex with $J_{+}(0)=0$, one has $F_{M}(v,u)\ge0$ for $0\le u\le M$ and
		$v\ge0$. Hence all three terms on the right–hand side of \eqref{eq:A.7} are non–negative and finite.
		Note that replacing $u$ by $u^{*}$ amounts to replacing $u_{M}$ and $v_{M}$ by $u^{*}_{M}$ and
		$v^{*}_{M}$, respectively. Below we shall prove that the double integral in \eqref{eq:A.7} does not increase
		if both $u_{M}$ and $v_{M}$ are replaced by $u^{*}_{M}$ and $v^{*}_{M}$. Moreover, by \eqref{eq:removebound},
		$I_t^{+}[v_{M}]\ge I_t^{+}[v^{*}_{M}]$. Hence if $I_t^{+}[u]=I_t^{+}[u^{*}]$, then $I_t^{+}[u_{M}]=I_t^{+}[u^{*}_{M}]$
		for all $M>0$. Using \eqref{eq:indiptau} one easily concludes that $u$ is the translate of a symmetric decreasing function.
		
		It suffices to prove that the double integral in \eqref{eq:A.7} does not increase under rearrangement.
		Since $J'_{+}$ is increasing, we have $J'_{+}(t)=\int_{0}^{t} d\mu(\tau)$ for a non–negative measure
		$\mu$. Hence
		\[
		J_{+}(t)=\int_{0}^{\infty} (t-\tau)_{+}\,d\mu(\tau)
		\quad\text{and}\quad
		F_{M}(v,u)=\int_{0}^{\infty} f_{M,\tau}(v,u)\,d\mu(\tau),
		\]
		where
		\[
		f_{M,\tau}(v,u):=(v+M-u-\tau)_{+}-(v-\tau)_{+}-(M-u-\tau)_{+}.
		\]
		Since the integrand is non–negative for $0\le u\le M$ and $v\ge0$, it suffices to prove that for
		almost every $\tau$ the double integral
		\[
		\iint_{\mathbb{R}^{N}\times\mathbb{R}^{N}}
		f_{M,\tau}\big(v_{M}(x),u_{M}(y)\big)e^{-t|x-y|^2}\,dx\,dy
		\]
		does not increase under rearrangement. We decompose further
		$f_{M,\tau}=f^{(1)}_{M,\tau}-f^{(2)}_{M,\tau}$ where
		\[
		f^{(1)}_{M,\tau}(v):=v-(v-\tau)_{+}
		\]
		and
		\[
		f^{(2)}_{M,\tau}(v,u):=v-(v+M-u-\tau)_{+}+(M-u-\tau)_{+}
		=\min\{\,v,\,(u-M+\tau)_{+}\,\}.
		\]
		Since $f^{(1)}_{M,\tau}$ is bounded and the support of $v_{M}$ has finite measure, the integral
		\[
		\iint_{\mathbb{R}^{N}\times\mathbb{R}^{N}} f^{(1)}_{M,\tau}\big(v_{M}(x)\big)\,e^{-t|x-y|^2}\,dx\,dy
		=\left\|e^{-t|\cdot|^2}\right\|_{L^1(\mathbb{R}^N)}\int_{\mathbb{R}^{N}} f^{(1)}_{M,\tau}\big(v_{M}(x)\big)\,dx
		\]
		is finite and invariant under rearrangement of $v_{M}$ for every $t>0$. Finally, by Fubini Theorem we can write
		\begin{equation}
			\label{eq:charnotbound}
			\begin{split}
				&\iint_{\mathbb{R}^{N}\times\mathbb{R}^{N}} f^{(2)}_{M,\tau}\big(v_{M}(x),u_{M}(y)\big)e^{-t|x-y|^2}\,dx\,dy =\\
				&\int_{0}^{\infty}
				\Bigg(\iint_{\mathbb{R}^{N}\times\mathbb{R}^{N}}
				\chi_{\{v_{M}>\sigma\}}(x)\,e^{-t|x-y|^2}\,\chi_{\{(u_{M}-M+\tau)_{+}>\sigma\}}(y)\,dx\,dy\Bigg)\,d\sigma.
			\end{split}
		\end{equation}
		By Lemma \ref{lem:rri}, the right-hand side in \eqref{eq:charnotbound} does not decrease under rearrangement for almost every $\tau>0$ and for every $t>0$, which completes the proof.
		
	\end{proof}
	
	\begin{remark}
		In \cite[Theorem 1]{FraSei08}, the fractional P\'olya-Szeg\"o inequality is set for $1\le p<\frac Ns$ in the homogeneous fractional Sobolev space $\dot{W}^{s,p}(\mathbb{R}^N):=\overline{C^\infty_c(\mathbb{R}^N)}^{\left[\cdot\right]_{W^{s,p}(\mathbb{R}^N)}}=\left\{u\in L^{\frac{Np}{N-sp}}(\mathbb{R}^N)\,;\,[u]_{W^{s,p}(\mathbb{R}^N)}<\infty\right\}$.
	\end{remark}
	
	Inequality \eqref{eq:rearrangementIt} is a particular case of a more general result proved in \cite[Lemma A.2]{FraSei08}.
	
	\begin{lemma}
		Let $J:\mathbb{R}\rightarrow [0,\infty)$ be a convex function with $J(0)=0$, $k\in L^1(\mathbb{R}^N)$ be a symmetric decreasing function and 
		$$
		E[u]:=\int_{\mathbb{R}^N}\int_{\mathbb{R}^N}J(u(x)-u(y))K(x-y)dxdy.
		$$
		Then, for every non-negative and measurable $u$ with $E[u]$ and $|\{u>\tau\}|$ finite for every $\tau>0$ one has that
		\begin{equation}
			\label{eq:generalrearr}
			E[u]\ge E[u^*].
		\end{equation}
	\end{lemma}
	
	\begin{remark}
		Inequality \eqref{eq:generalrearr} has been proved also in \cite[Theorem 9.2 (i)]{AlmLie} under the additional hypoteses of even symmetry of $J$ and finiteness of the quantity $\int_{\mathbb{R}^N}J(u(x))dx$. Clearly these hypoteses encompass also the $W^{s,p}(\mathbb{R}^N)$ seminorms thanks to the representation formula \eqref{eq:heatrepr}. The novelty of the proof given in \cite{FraSei08} is the \emph{rigidity}, namely, the characterization of the cases of equality in \eqref{eq:generalrearr}.
	\end{remark}

	\section{Some applications}
	\label{sec:applications}
	\subsection{Fractional Sobolev inequality}
	\label{sec:fracsoboine}
	Let $1\le p<\frac Ns$. The \emph{fractional Sobolev inequality} states that the continuous embedding
	\begin{equation}
		\label{eq:fracsobine}
	W^{s,p}(\mathbb{R}^N)\hookrightarrow L^q(\mathbb{R}^N)
	\end{equation}
	holds true for every $q\in[p,p^*_s]$, where $p^*_s:=\frac{Np}{N-sp}$ is the \emph{fractional Sobolev critical exponent}. In particular, in the critical case $q=p^*_s$ one can search for a maximizer of the Sobolev quotient
	\begin{equation}
	\label{eq:Sobquotient}
	\mathcal{S}_{s,p}(u):=\frac{\left\|u\right\|_{L^{p^*_s}(\mathbb{R}^N)}}{[u]_{W^{s,p}(\mathbb{R}^N)}}
	\end{equation}
	among all $u\in W^{s,p}(\mathbb{R}^N)\setminus\{0\}$. Through the P\'olya-Szeg\"o inequality one can reduce to the one-dimensional case searching for a radial maximizer to \eqref{eq:Sobquotient}. In particular, when $p=2$, the unique optimal profile up to translations and dilations is explicit and it is given by
	$$
	u_{\rm max,2}(x):=(1+|x|^2)^{-\frac{N-2s}{2}},
	$$
	see \cite[Theorem 1.1]{CotTav}, while for $p\ne 2$ it is only known that the optimal profile $u_{\rm max,p}$ is radially decreasing with $u_{\rm max,p}(x)\sim|x|^{-\frac{N-sp}{p-1}}$ as $|x|\to\infty$ (see \cite[Theorem 1.1.]{BraMosSqu}). It is worth noticing that for $s\to 1^-$ we recover the behaviour of the maximizers of the classical Sobolev quotient $\mathcal{S}_{1,p}(u)$ commonly known as \emph{Aubin-Talenti bubbles}.
	\subsection{Fractional Isoperimetric inequality}
	\label{sec:isopine}
	Nonlocal perimeters either of fractional type or depending on more general kernels have been widely studied in the last years, since they are related to nonlocal minimal surfaces \cite{CafRoqSav, MazRosTol}, phase transitions \cite{valdinoci}, fractal sets \cite{lombardini} and many other problems. The relevant isoperimetric inequalities of qualitative and quantitative type have been proved in \cite{CesNov, FraSei08} and \cite{FuMiMo, FiFuMaMiMo}, respectively. See also \cite{ComSte} where the authors introduce a notion of fractional perimeter using a distributional approach and \cite{DiNoRuVa} where an isoperimetric problem with the competition of two fractional perimeters of different order is studied.
	
	The $s$-fractional perimeter $P_s(E)$ of a measurable set $E \subset \mathbb{R}^N$ is defined as:
	\begin{equation}
		P_s(E) = \frac{[\chi_E]_{W^{s,1}(\mathbb{R}^N)}}{2} = \int_E\int_{E^c} \frac{dx dy}{|x-y|^{N+s}},
	\end{equation}
	in particular, $E$ is a set of finite $s$-fractional perimeter if and only if $\chi_E\in W^{s,1}(\mathbb{R}^N)$.
	
	The fractional P\'olya-Szeg\"o principle yields the following
	\begin{theorem}[Fractional Isoperimetric Inequality]
		Among all open and bounded sets, balls minimize the scale invariant isoperimetric ratio, i.e.
		\begin{equation}
			\label{eq:fracisoine}
			\frac{P_s(\Omega)}{|\Omega|^{\frac{N-s}{N}}}\ge \frac{P_s(B)}{|B|^{\frac{N-s}{N}}}.
		\end{equation}
		In particular, if $|\Omega|=|B|$ then
		$$
		P_s(\Omega) \ge P_s(B),
		$$
		and equality holds if and only if $\Omega$ is a ball.
	\end{theorem}
	\begin{proof}
		It is sufficient to apply Theorem \ref{th:frpsi} with $u=\chi_\Omega$ and using that $\chi^*_\Omega=\chi_{\Omega^*}.$
	\end{proof}
	
	\begin{remark}
		Inequality \eqref{eq:fracisoine} can be restated by noticing that $P_s(B)=C|B|^{\frac{N-s}{N}}=C|\Omega|^{\frac{N-s}{N}}$ for some explicit $C=C(N,s)$, (see \cite[Proposition 1.1.]{Garofalo2} replacing $s$ with $\frac s2$). In particular, this yields that for every open set $\Omega$ with finite measure it holds
		$$
		P_s(\Omega)\ge C|\Omega|^{\frac{N-s}{N}}.
		$$
		This is also a remarkable application of the fractional Sobolev inequality \eqref{eq:fracsobine} to $\chi_\Omega$.
	\end{remark}
	
	\begin{remark}
		Since $|\chi_E(x)-\chi_E(y)|=|\chi_E(x)-\chi_E(y)|^p$ for every $x,y\in\mathbb{R}^N$ and $1\le p<\infty$ it is easy to check that
		$$
		P_s(E)=\frac12[\chi_E]^p_{W^{\frac sp,p}(\mathbb{R}^N)}.
		$$
	\end{remark}
	
	%	\begin{remark}
		%	Similar techniques apply to fractional Cheeger inequalities, where symmetrization yields a lower bound on the first Dirichlet eigenvalue of fractional laplacian.
		%	\end{remark}

	\subsection{Fractional Dirichlet eigenvalues and Faber-Krahn Inequality}
	\label{sec:fracfk}
	Let $1<p<\frac Ns$ and $\Omega$ be a bounded open set in $\mathbb{R}^N$. The Dirichlet eigenvalue problem for $(-\Delta)^s_p$ consists in finding the minimum value $\lambda$ such that the external value problem
	\begin{equation}
		\label{eq:firstdiregienfun}
		\begin{cases}
			(-\Delta)^s_pu=\lambda|u|^{p-2} u\quad\text{in}\quad\Omega \\
			u=0\quad\text{in}\quad\mathbb{R}^N\setminus\Omega
		\end{cases}
	\end{equation}
	admits a nonzero weak solution.
	
	Such a minimum admits a variational characterization in terms of the \emph{Rayleigh quotient}. We define
	$$
	W_0^{s,p}(\Omega):=\overline{C^\infty_c(\Omega)}^{\left\|\cdot\right\|_{W^{s,p}(\mathbb{R}^N)}}\subsetneq W^{s,p}(\mathbb{R}^N)
	$$
	and
	$
	\mathcal{R}_{\Omega,s,p}:W^{s,p}_0(\Omega)\rightarrow[0,\infty)
	$ defined as
	\begin{equation}
		\label{eq:rayleighquotient}
		\mathcal{R}_{\Omega,s,p}(u):=\frac{[u]^p_{W^{s,p}(\mathbb{R}^N)}}{\left\|u\right\|^p_{L^p(\mathbb{R}^N)}}
	\end{equation}
	The first Dirichlet eigenvalue for $(-\Delta)^s_p$ is therefore characterized in the following way
	\begin{equation}
		\lambda_1^{s,p}(\Omega) = \inf_{u \in W_0^{s,p}(\Omega)\setminus\{0\}} \mathcal{R}_{\Omega,s,p}(u).
	\end{equation}
	\begin{remark}
		Using Direct Methods of the calculus of variations it can be proved that the infimum is actually a minimum and it is attained in $u=u_\Omega$ where $u_\Omega$ is the first Dirichlet eigenfunction, i.e., the unique weak solution of \eqref{eq:firstdiregienfun}.
	\end{remark}
	\begin{remark}
		Notice that the right-hand side in \eqref{eq:rayleighquotient} actually depends on $\Omega$. Indeed, since $u=0$ in $\mathbb{R}^N\setminus\Omega$ we have that
		$$
		\mathcal{R}_{\Omega,s}(u)=\frac{\mathcal{I}_{\Omega,s,p}(u)}{\left\|u\right\|^p_{L^p(\Omega)}},
		$$
		where
		$$\mathcal{I}_{\Omega,s,p}(u):=\iint_{\mathcal{Q}_\Omega}\frac{|u(x)-u(y)|^p}{|x-y|^{N+sp}}dxdy$$
		and
		$
		\mathcal{Q}_\Omega:=(\Omega\times\Omega)\cup(\Omega\times\Omega^c)\cup(\Omega^c\times\Omega).
		$
	\end{remark}
	
	The fractional Faber-Krahn inequality goes as follows, see \cite[Theorem 3.5]{BraLinPar}:
	\begin{theorem}[Fractional Faber--Krahn Inequality]
		Among all open and bounded sets balls minimize the scale invariant functional:
		\begin{equation}
			|\Omega|^{\frac{sp}{N}}\lambda_1^{s,p}(\Omega) \ge |B|^{\frac{sp}{N}}\lambda_1^{s,p}(B).
		\end{equation}
		In particular, if $|B|=|\Omega|$ we have the Faber-Krahn inequality for the first Dirichlet eigenvalue of $(-\Delta)^s_p$
		$$
		\lambda_1^{s,p}(\Omega) \ge\lambda_1^{s,p}(B),
		$$
		and equality holds if and only if $\Omega$ is a ball.
	\end{theorem}
	
	\begin{proof}
		The proof follows by Theorem \ref{th:frpsi} by using that $\lambda_1^{s,p}(\Omega)=\mathcal{R}_{\Omega,s,p}(u_\Omega)=[u_\Omega]^p_{W^{s,p}(\mathbb{R}^N)}$, where $u_\Omega$ is the first Dirichlet eigenfunction with $L^p(\Omega)$ unit norm, and that $u^*_\Omega=u_{\Omega^*}$.
	\end{proof}
	
	\subsection{Talenti-type comparison for solutions to fractional equations}
	\label{sec:talenti}
	The P\'olya-Szeg\"o principle allows also in the fractional case for a comparison between a solution to a fractional equation and a solution to a corresponding symmetrized problem. For the sake of concreteness we state the Theorem in the model case of the fractional laplacian. We refer to \cite[Theorems 1.1, 1.2]{FerPisVol} where a broad class of fractional elliptic and parabolic operator is also treated.
	
	\begin{theorem}
		\label{th:masscompfpv}
		Let $\Omega\subset\mathbb{R}^N$ be an open and bounded set and $f\in L^2(\Omega)$, $f\ge 0$. If $u$ and $v$ are respectively solutions to
		$$
		\begin{cases*}
			(-\Delta)^su=f\quad\text{in}\quad\Omega \\
			u=0\quad\text{in}\quad\mathbb{R}^N\setminus\Omega
		\end{cases*}
		$$
		and
		$$
		\begin{cases*}
			(-\Delta)^sv=f^*\quad\text{in}\quad\Omega^* \\
			v=0\quad\text{in}\quad\mathbb{R}^N\setminus\Omega^*,
		\end{cases*}
		$$
		Then
		$$
		\int_{B_r}u^*(x)dx\le\int_{B_r}v(x)dx\quad\forall r>0
		$$
		and
		$$
		[u]_{H^s(\mathbb{R}^N)}\le[v]_{H^s(\mathbb{R}^N)}.
		$$
	\end{theorem}
		
	\section{P\'olya-Szeg\"o inequalities via partial symmetrization of extension functionals}
	\label{sec:pspviaext}
	We now show in the special case $p=2$ an alternative proof of Theorem \ref{th:frpsi} based on the characterization of the $H^s$-seminorm given by Theorem \ref{th:cafsil} and on the \emph{Partial Symmetrization} of the weighted Dirichlet integral in \eqref{eq:extseminorm}. We give here a brief sketch of the proof

	\begin{enumerate}
		\item We set
		\begin{equation}
		\label{eq:I1}
		I_1[U]:=\int_0^\infty y^{1-2s}\left(\int_{\mathbb{R}^N}|\nabla_x U(x,y)|^2dx\right)dy
		\end{equation}
		and
		\begin{equation}
		\label{eq:I2}
		I_2[U]:=\int_0^\infty y^{1-2s}\left(\int_{\mathbb{R}^N}|\partial_y U(x,y)|^2dx\right)dy
		\end{equation}
		where $U$ solves \eqref{eq:divformeq}.

		\item With a slight abuse of notation we indicate again with $U^*$ the Schwarz symmetrization of $U$ with respect the $x$ variable. Therefore the classical P\'olya-Szeg\"o inequality \eqref{eq:classicalps} yields that
		$$
		\int_{\mathbb{R}^N}|\nabla_xU^*(x,y)|^2dx\le\int_{\mathbb{R}^N}|\nabla_xU(x,y)|^2dx
		$$
		for every $y>0$. Multiplying by $y^{1-2s}$ and integrating in $(0,\infty)$ we obtain that $I_1[U^*]\le I_1[U]$.
		\item The inequality
		\begin{equation}
			\int_{\mathbb{R}^N}|\partial_y U^*(x,y)|^2dx\le\int_{\mathbb{R}^N}|\partial_y U(x,y)|^2dx
		\end{equation}
		is a consequence of \cite[Theorem 1]{brock}. Again, by multiplying by $y^{1-2s}$ and integrating in $(0,\infty)$ we obtain that $I_2[U^*]\le I_2[U]$.
		\item Let $V:\mathbb{R}^{N+1}_+\rightarrow [0,\infty)$ the unique weak solution of
		\begin{equation}
			\label{eq:divformeqsymm}
			\begin{cases}
				\text{div}(y^{1-2s}\nabla V)=0\quad\text{in}\quad \mathbb{R}^{N+1}_+ \\
				V(x,0)=u^*(x)\quad\text{in}\quad\mathbb{R}^N.
			\end{cases}
		\end{equation}
		It can be proved (see\cite[Proposition 3.2]{BraCinVit}) that $V$ has less energy than $U^*$, namely
		\begin{equation}
			\label{eq:masscon}
			I_1[V]+I_2[V]\le I_1[U^*]+I_2[U^*],
		\end{equation}
		where $U$ solves \eqref{eq:divformeq}. Therefore
		$$
		[u^*]^2_{H^s(\mathbb{R}^N)}=I_1[V]+I_2[V]\le I_1[U^*]+I_2[U^*] \le I_1[U]+I_2[U]=[u]^2_{H^s(\mathbb{R}^N)}
		$$
		and this proves Theorem \ref{th:frpsi}.
	\end{enumerate}
	
	\begin{remark}
		By replacing $s$ with $\frac s2$ and the boundary datum $u$ with the characteristic function of some open and bounded set $\Omega$ in \eqref{eq:extseminorm}, the following characterization of the fractional perimeter can be used
					\begin{equation}
			\label{eq:extperimeter}
			P_s(\Omega)=c_{\frac s2}\inf_{\substack{U\in H^1_{\rm loc}(\mathbb{R}^{N+1}_+) \\ U\left(\cdot,0\right)=\chi_\Omega}}\int_0^\infty y^{1-s}dy\int_{\mathbb{R}^N}\left(|\nabla_xU|^2+|\partial_yU|^2\right)dx
		\end{equation}
		and the fractional isoperimetric inequality through the classical P\'olya-Szeg\"o principle can also be achieved. The same representation formula holds for the first Dirichlet fractional eigenvalue just by using the first eigenfunction of $(-\Delta)^s$ as Dirichlet datum in the extension problem \eqref{eq:divformeq}.
	\end{remark}
	
		\section{Stability of some Geometric Inequalities in the Fractional Setting}
	\label{sec:stability}
	Quantitative versions of functional inequalities aim to measure their \emph{stability}, namely how far a given function or set is from attaining equality in the corresponding sharp inequality. In the fractional framework, such results are subtle due to the nonlocal nature of the underlying functionals. In this section, we review some noticeable results concerning sharp quantitative inequalities associated with nonlocal energies.
	
	Concerning the fractional isoperimetric inequality, a quantitative enhancement involves the Fraenkel asymmetry
	\[
	\mathcal{A}(\Omega) := \inf_{x \in \mathbb{R}^N} \frac{|\Omega \Delta B_r(x)|}{|\Omega|}, \quad \text{where } |B_r| = |\Omega|:
	\]
	given some geometric functional $\mathcal{F}$ and $m\in(0,\infty)$, once proved that the problem
	$$
	\min_{|\Omega|=m}\mathcal{F}(\Omega)
	$$
	admits a (unique) solution given by the ball, the goal is to prove that, under the measure constraint, the deficit $D(\Omega):=\mathcal{F}(\Omega)-\mathcal{F}(B)$ can be bounded from below by some positive function $g$ of the Fraenkel asymmetry, namely
	$$
	D(\Omega)\ge g(\mathcal{A}(\Omega)).
	$$ A very flexible technique in the local setting is the \emph{Selection Principle} introduced in \cite{CicLeo}, which provides a way to extract a sequence $(E_h)_h$ of nearly optimal sets for a quantitative geometric inequality from arbitrary minimizing sequences. Then, using compactness and regularity results, one shows that $(E_h)_h$ converges (after suitable translations) to a ball, in the sense that
	$$
	\mathcal{A}(E_h)\to 0\quad\text{and}\quad\frac{D(E_h)}{g\left(\mathcal{A}(E_h)\right)}\to\inf_{\Omega}\frac{D(\Omega)}{g\left(\mathcal{A}(\Omega)\right)},
	$$ and satisfies uniform estimates of the second variations, which leads to the desired inequality.
	 These tools have been adapted to the fractional perimeter in \cite{FiFuMaMiMo} obtaining the following quantitative form of \eqref{eq:fracisoine} 
	\begin{theorem}
		For every $N\ge 2$ and $s_0\in(0,1)$ there exists a positive constant $C_{N,s_0}>0$ such that
		\begin{equation}
			\label{eq:ffmmm}
			P_s(\Omega) \ge P_s(B) \left(\frac{|\Omega|}{|B|}\right)^{\frac{N-s}{N}}\left(1+\frac{\mathcal{A}(\Omega)^2}{c_{N,s_0}}\right),
		\end{equation}
		for every $s\in[s_0,1]$ and for every $0<|\Omega|<\infty$. In particular, if $|\Omega|=|B|$, then
		$$
		D_s(\Omega):=P_s(\Omega)-P_s(B)\ge\frac{P_s(B)}{C_{N,s_0}}\mathcal{A}(\Omega)^2.
		$$
	\end{theorem}
	
	\begin{remark}
		Some comments are in order. It can be proved via a counterexample that the power 2 in the right-hand side of \eqref{eq:ffmmm} is sharp. Nevertheless, the Selection principle does not allow to find an explicit constant $C_{N,s_0}$, though it is conjectured that $C_{N,S_0}\approx\frac{C'_N}{s_0}$ as $s_0\to 0^+$.
	\end{remark}
	
	An alternative proof for quantitative isoperimetric inequalities exploits the extension technique stated in Theorem \ref{th:cafsil}. In this framework, one can deduce a stability estimate basically using:
	\begin{itemize}
		\item the characterization of the $H^s(\mathbb{R}^N)$-seminorm \eqref{eq:extseminorm} \\
		\item the sharp quantitative isoperimetric inequality by Fusco-Maggi-Pratelli \cite{FuMaPr} \\
		\item  the so-called \emph{transfer of asymmetry} which allows to estimate from below the asymmetry of the superlevel sets $E_{t,y}:=\{x\in\mathbb{R}^N\,;U_\Omega(x,y)>t\}$ of the extension $U_\Omega$ of the eigenfunction $u_\Omega$ for suitable levels $t>0$ and suitable values of the extension variable $y\in(0,\infty)$ with a multiple of the asymmetry of $\Omega$.
	\end{itemize}
With this very flexible technique quantitative isoperimetric inqualities for the fractional perimeter in \cite{FuMiMo}, for semilinear fractional Dirichlet eigenvalues \cite{BraCinVit} and for the fractional 2-capacity \cite{CiOgRu} have been obtained. Though the proof exploits the sharp quantitative isoperimetric inequality, the rearrangement of those inequalities on some extension levels does not lead to the sharp power of the Fraenkel asymmetry.
	
		\section{Some generalizations}
		\label{sec:anisotrop}
		In this section we treat some generalizations when the relevant energy is not necessarily isotropic. Anisotropic energies reflect the influence of a non-Euclidean geometry on the variational structure of the problem. In this setting, symmetrization no longer leads to the euclidean ball as optimal shape but rather to Wulff shapes or other equivalent geometric structures, highlighting how the optimal shape inherits the underlying anisotropy.
		\subsection{The anisotropic case}
		The symmetrization of functionals of the form
		$$
		\int_{\Omega}A(x,H(\nabla u)) dx
		$$
		for some norm $H:\mathbb{R}^N\rightarrow[0,\infty)$ and some Young function $A:\Omega\times [0,\infty)\rightarrow [0,\infty)$ has been addressed by several authors. We refer e.g. to the book \cite{vanschaft}.
		
		A meaningful example is given by $A(x,H(\xi))=H_K(\xi)$ for some norm $H_K$ associated to a convex body $K\subset\mathbb{R}^N$.  
		Moving to the fractional framework, anisotropic fractional seminorms have been introduced in \cite{Ludwig}, while the corresponding P\'olya-Szeg\"o principle with respect to the Steiner symmetrization defined in Section \ref{sec:notprel} for the relevant $W^{s,p}$ seminorm
		$$
		[f]_{W_K^{s,p}(\mathbb{R}^N)}:=\left(\int_{\mathbb{R}^N}\int_{\mathbb{R}^N} \frac{|u(x)-u(y)|^p}{\left\|x-y\right\|_K^{N+sp}}\,dx\,dy\right)^{1/p}
		$$
		has been proved in \cite{Kreuml} deducing an isoperimetric inequality for the relevant anisotropic fractional perimeter with minimizer given by the \emph{Wulff shape} induced by $K$.
		\subsection{The Gaussian case}
		Another significant example is given by $A(x,H(\xi))=w(x)|\xi|^p$, for some $1\le p<\infty$ and some positive weight $w$. In particular for $w=\gamma:=\frac{e^{-\frac{|\cdot|^2}{2}}}{(2\pi)^{N/2}}$ we deal with the Gaussian rearrangement given by the \emph{Ehrhard symmetrization} introduced in \cite{EhrScand}. Through this symmetrization technique which rearranges a set in a halfspace keeping fixed its Gaussian volume, the same author proves the relevant P\'olya-Szeg\"o inequality \cite[Th\'eorème 3.1]{ehrhard} and deduces the isoperimetric inequality \cite[Proposition 2.1]{ehrhard} and the Faber-Krahn inequality \cite[Th\'eorème 4.4]{ehrhard}. See also \cite{BarBraJul, CCLP4} for quantitative estimates.
		
		In the nonlocal counterpart the notion of $H^s$-fractional Gaussian seminorm has been introduced in \cite{NovPalSir} in the setting of abstract Wiener spaces through an extension technique \emph{à la} Caffarelli-Silvestre proved in \cite{StiTor}, and the authors prove an isoperimetric inequality for the \emph{fractional Gaussian perimeter} $P^\gamma_s(E)$ using the characterization in formula \eqref{eq:extperimeter} replacing $(\mathbb{R}^N,\mathcal{L}^N)$ with a Wiener space $X$ equipped with the Gaussian measure $\gamma$. A stability result has also been obtained in \cite{CCLP4} in finite dimension.
		
		It is worth noticing that in the Gaussian framework the formulation with the extension technique is still equivalent with the one through Bochner subordination by considering the fractional powers of the (negative) \emph{Ornstein-Uhlenbeck operator} $-\Delta_\gamma$, where $\Delta_\gamma:=\Delta-\langle x,\nabla\rangle$ instead of the laplacian. See \cite{StiTor}. 
		
		The usefulness of this approach is manifold; comparison results as in Theorem \ref{th:masscompfpv} for solutions to equations driven by fractional powers of $-\Delta_\gamma$ can be proved (\cite{FeStVo}). Moreover, it simplifies the study of $W^{s,p}_\gamma$ fractional seminorms \emph{\'a la} Gagliardo or the analysis of the fractional Gaussian perimeter of a set in a bounded domain, see \cite{CCLP1, CCLP2}.
		
		Another, non-equivalent, definition of fractional perimeter in the Gauss space has been given in \cite{DL}. In this work has been proved that the only volume constrained isoperimetric sets are halfspaces passing through the origin, proving this way that for the given definition of Gaussian fractional perimeter the Ehrhard symmetrization may in general be not convenient.
	
	\section{Final comments and Open problems}
	\label{sec:openprob}
	Symmetrization methods are invaluable tools also in the fractional setting, allowing the derivation of sharp inequalities involving nonlocal geometric functionals. We conclude this chapter with some open problems
	\begin{enumerate}
		\item \textbf{Sharp quantitative inequalities}
		Many inequalities stated in Section \ref{sec:stability} are conjectured to be not sharp. To improve them the use of the Selection Principle as done in \cite{FiFuMaMiMo}, seems unavoidable. \\
		\item \textbf{P\'olya-Szeg\"o inequalities for Dirichlet integrals of the Riesz fractional gradient} In \cite[Section 4]{NguSqu} the authors set as open problem the validity or not of the inequality
			\begin{equation}
				\label{eq:psrfg}
			\int_{\mathbb{R}^N}|\nabla^s u^*|^pdx\le\int_{\mathbb{R}^N}|\nabla^s u|^pdx.
		\end{equation}
		being
		$$
		\nabla^su(x):=\mu_{N,s}\int_{\mathbb{R}^N}\frac{(u(x)-u(y))(x-y)}{|x-y|^{N+s+1}}dy
		$$
		the \emph{Riesz fractional gradient}, in the potential space $H^{s,p}(\mathbb{R}^N):=\{u\in L^p(\mathbb{R}^N):\,\nabla^su\in L^p(\mathbb{R}^N,\mathbb{R}^N)\}$. We notice that in general the equality
		$$
		\left\|\nabla^su\right\|_{L^p(\mathbb{R}^N,\mathbb{R}^N)}=C_{N,s,p}[u]_{W^{s,p}(\mathbb{R}^N)}
		$$ 
		holds only for $p=2$, and $H^{s,2}(\mathbb{R}^N)=H^s(\mathbb{R}^N)$. Then, for $p=2$ inequality \eqref{eq:psrfg} immediately follows from Theorem \ref{th:frpsi}, but for $p\ne 2$ neither the Schwarz symmetrization nor the Steiner symmetrization are not known to be useful. See e.g. \cite{ComSte} for further properties of $H^{s,p}(\mathbb{R}^N)$ and $\nabla^s$. \\
%		A proof of \eqref{eq:psrfg} could be obtained adapting the polarization rearrangement enjoyed by Gagliardo seminorms, see e.g. the book \cite{baernstein}.
		\item \textbf{Fractional isoperimetric inequalities in the sub-Riemannian setting} 
		
		One of the most celebrated open problems in sub-Riemannian geometry concerns the identification of the optimal isoperimetric profile. In the Heisenberg group $\mathbb{H}^N$, this question is known as \emph{Pansu’s conjecture}, which remains open despite several partial results. In contrast, for $\alpha\ge 0$ the metric measure space $(\mathbb{R}^{N+1}, d,\mathcal{L}^{N+1})$ with $ds^2=|dx|^2+\frac{dy^2}{(\alpha+1)|x|^{2\alpha}}$ is commonly known as \emph{Grushin space} and has an explicit isoperimetric profile which has been determined in \cite{FraMon, MonMor}.
		
		A natural question is to define a \emph{fractional} analogue of the intrinsic distributional perimeter in Grushin spaces. A possible approach could rely on an extension procedure associated with the intrinsic laplacian given by the (negative) \emph{Grushin operator} $-\mathcal{G}_\alpha$, where
		\[
		\mathcal{G}_\alpha u := \Delta_x u + (\alpha+1)|x|^{2\alpha} \partial_{yy} u,
		\]
		and on an adaptation of the Caffarelli--Silvestre-type extension to deal with the energy functional
		\[
		\int_0^\infty z^{1-s}dz\int_{\mathbb{R}^{N+1}}|\nabla^\alpha_{x,y}U(x,y,z)|^2+|\partial_z U(x,y,z)|^2dxdy,
		\]
		where
		$$
		|\nabla^\alpha_{x,y}U(x,y,z)|^2=|\nabla_{x}U(x,y,z)|^2+(\alpha+1)^2|x|^{2\alpha}|\partial_{y}U(x,y,z)|^2.
		$$
		
		In \cite{beckner}, the author observes that the Grushin metric is conformal to the Poincaré one through a change of variables. When endowed with the Poincarè metric the upper-half space $\mathbb{R}^{N+1}_+$ is known to be a model of hyperbolic geometry. In this setting the \emph{hyperbolic symmetrization} rearranges the super-level sets of a function into metric balls keeping fixed the volume, so that it plays the role of the Schwarz symmetrization in the isotropic case and decreases the integral
		$$
		\int_{\mathbb{R}^{N+1}}|\nabla^\alpha_{x,y}U(x,y,z)|^2dxdy
		$$
		for every $z>0$, thus yielding that also the term
		$$
		J_1[U]:=\int_0^\infty z^{1-s}dz\int_{\mathbb{R}^{N+1}}|\nabla^\alpha_{x,y}U(x,y,z)|^2dxdy
		$$
		decreases applying the partial hyperbolic symmetrization with respect to the variable $(x,y)$ to $U$. However, it is not clear how to deal with the term
		$$
		J_2[U]:=\int_0^\infty z^{1-s}dz\int_{\mathbb{R}^{N+1}}|\partial_z U(x,y,z)|^2dxdy,
		$$
		for which an intrinsic version of the Steiner symmetrization (as in \cite{brock}) would be required.
		
		A rigorous implementation of this strategy could lead to a fractional isoperimetric inequality in the Grushin setting, yielding a different isoperimetric profile with respect to the one obtained for the classical distributional perimeter in \cite{FraMon, MonMor}.

%		One of the most famous open problems in sub-riemannian geometry concerns the identification of the relevant isoperimetric profile. In the Heisenberg group $\mathbb{H}^N$ this problem is known as \emph{Pansu's conjecture} which is still unsolved though partial answers have been given. In the Grushin space the isoperimetric profile has been computed in \cite{FraMon, MonMor}. The minimization of a fractional counterpart of the intrinsic distributional perimeter could be possible via extension procedure of the intrinsic laplacian given by the \emph{Grushin operator} $-\mathcal{G}_\alpha u:=-\Delta_xu-(\alpha+1)|x|^{2\alpha}\Delta_yu$, using the \emph{hyperbolic rearrangement} (see \cite{beckner}) to decrease the term \eqref{eq:I1} with the intrinsic gradient. Nevertheless, it is not clear how to apply an intrinsic Steiner symmetrization as in \cite{brock} to treat the term \eqref{eq:I2}. It is worth noticing that a positive answer would give a different isoperimetric profile with respect to the one obtained for the distributional perimeter in \cite{FraMon, MonMor}.
	\end{enumerate}

	\begin{acknowledgement}
		The author kindly thanks Dr. Alberto Maione and Dr. Joaquim Duran for the opportunity of writing this book chapter. The author thanks also the Centre de Recerca Matematica for the very kind hospitality in March 2025.
		
		Most of the results in this chapter were the subject of a PhD course held by the author in January 2025 at the Dipartimento di Matematica e applicazioni ``Federico II'' of the University of Naples for which the author thanks Prof. Carlo Nitsch.
		
		The author also warmly thanks Prof. Diego Pallara and the anonymous referee for very useful suggestions in the revision of the chapter.
	\end{acknowledgement}
	\ethics{Funding informations}{The author is member of Gruppo Nazionale per l'Analisi Matematica, Probabilità e le loro Applicazioni of the Istituto Nazionale di Alta Matematica (INdAM-GNAMPA) and he acknowledges the support of the MUR - PRIN 2022 project ``Elliptic and parabolic problems, heat kernel estimates and spectral theory'' (N.~20223L2NWK) and of the INdAM - GNAMPA 2025 Project ``Metodi variazionali per problemi dipendenti da operatori frazionari isotropi e anisotropi'' (Grant Code: CUP\_E5324001950001). 
	The author has no conflicts of interest to declare that are relevant to the content of this chapter.}

\end{document}